\documentclass[12pt,a4paper,oneside]{amsart}

\usepackage{amsfonts, amsmath, amssymb, amsthm, hyperref}
\usepackage{anysize}

\newtheorem{thm}{Theorem}[section]
\newtheorem*{thm*}{Theorem}
\newtheorem{lem}[thm]{Lemma}

\newtheorem{fcon}[thm]{False Conjecture}

\theoremstyle{definition}
\newtheorem{defn}[thm]{Definition}

\theoremstyle{remark}
\newtheorem{rem}[thm]{Remark}

\numberwithin{equation}{section}

 \DeclareMathOperator{\cl}{cl}

\newcommand{\inte}{\mathop{\rm int}}

\renewcommand{\epsilon}{\varepsilon}
\renewcommand{\phi}{\varphi}

\newcommand{\Sg}{\Sigma}

\begin{document}
\title{Equipartition of several measures}
\author{R.N.~Karasev}
\thanks{This research is supported by the Dynasty Foundation, the President's of Russian Federation grant MD-352.2012.1, the Russian Foundation for Basic Research grants 10-01-00096 and 10-01-00139, the Federal Program ``Scientific and scientific-pedagogical staff of innovative Russia'' 2009--2013, and the Russian government project 11.G34.31.0053.}

\email{r\_n\_karasev@mail.ru}
\address{Dept. of Mathematics, Moscow Institute of Physics and Technology, Institutskiy per. 9, Dolgoprudny, Russia 141700}
\address{Institute for Information Transmission Problems RAS, Bolshoy Karetny per. 19, Moscow, Russia 127994}
\address{Laboratory of Discrete and Computational Geometry, Yaroslavl' State University, Sovetskaya st. 14, Yaroslavl', Russia 150000}


\begin{abstract}
We prove several results of the following type: any $d$ measures in $\mathbb R^d$ can be partitioned simultaneously into $k$ equal parts by a convex partition (this particular result is proved independently by Pablo~Sober\'on). Another example is: Any convex body in the plane can be partitioned into $q$ parts of equal areas and perimeters provided $q$ is a prime power.

The above results give a partial answer to several questions posed by A.~Kaneko, M.~Kano, R.~Nandakumar, N.~Ramana Rao, and I.~B\'{a}r\'{a}ny. The proofs in this paper are inspired by the generalization of the Borsuk--Ulam theorem by M.~Gromov and Y.~Memarian.

The main tolopogical tool in proving these facts is the lemma about the cohomology of configuration spaces originated in the work of V.A.~Vasil'ev.

A newer version of this paper, merged with the similar paper of A.~Hubard and B.~Aronov is \href{http://arxiv.org/abs/1306.2741}{arXiv:1306.2741}.
\end{abstract}

\maketitle

\section{Introduction}

We use the idea in Gromov's generalization of the Borsuk--Ulam theorem from~\cite{gr2003,mem2009} to prove the theorem on simultaneous partitioning of several measures into equal parts. More generally, we partition several measures into equal parts and require several continuous (in a certain sense) functions of these parts coincide.

Let us make some definitions. Consider a compact topological space $X$ with a Borel probability measure $\mu$. Let $C(X)$ denote the set of continuous functions on $X$.

\begin{defn}
A finite-dimensional linear subspace $L\subset C(X)$ is called \emph{measure separating}, if for any $f\neq g\in L$ the measure of the set 
$$
e(f, g) = \{x\in X : f(x) = g(x)\}
$$
is zero.
\end{defn}

In particular, if $X$ is a compact subset of $\mathbb R^n$ such that $X=\cl(\inte X)$ and $\mu$ is any absolutely continuous measure then any finite-dimensional space of analytic functions is measure-separating, because the sets $e(f,g)$ always have dimension $<n$ and therefore measure zero. Then for any collection of $q$ elements of a measure-separating subspace we define a partition of $X$.

\begin{defn}
Suppose $F=\{u_1, \ldots, u_q\} \subset C(X)$ is a family of functions such that $\mu (e(u_i,u_j)) = 0$ for all $i\neq j$. The sets (some of them may be empty)
$$
V_i = \{x\in X : \forall j\neq i\ u_i(x)\ge u_j(x)\}
$$
have a zero measure overlap, so they define a partition $P(F)$ of $X$. In case $u_i$ are linear functions on $\mathbb R^n$ we call $P(F)$ a \emph{generalized Voronoi partition}.
\end{defn}

Now we are ready to state the result about partitioning a measure into equal parts.

\begin{thm}
\label{gen-spl}
Suppose $L$ is a measure-separating subspace of $C(X)$ of dimension $n+1$, $\mu_1,\ldots, \mu_n$ are absolutely continuous (with respect to the original measure on $X$) probability measures on $X$. Then for any prime power $q$ there exists a $q$-element subset $F\subset L$ such that for every $i=1,\ldots,n$ the partition $P(F)$ partitions the measure $\mu_i$ into $q$ equal parts.
\end{thm}

The ``ham sandwich'' theorem~\cite{st1942,ste1945} follows from this theorem, is we let $X=\mathbb R^n$, $L$ be the space of polynomials of degree $\le 1$, and $q=2$. Moreover, taking $q>2$ in this theorem we obtain partitions of $\mathbb R^n$ into $q$ convex parts, partitioning every measure $\mu_1, \ldots, \mu_n$ into $q$ equal parts. This is true for prime powers $q$, and for arbitrary $q$ it can be obtained by iterating partitions, though the partitions will no more be generalized Voronoi partitions. Such results about convex partitions of the plane were known, see~\cite{kbs2000} for example.

Some results similar to Theorem~\ref{gen-spl} were independently obtained in~\cite{arhu2010, sob2010}, in those papers $q$ was prime and the functions were linear (so the partition was convex). After the discussions between the author and the authors of~\cite{arhu2010} the second version of~\cite{arhu2010} was updated to include the prime power case. The reader may also find in the second version of~\cite{arhu2010} an analogue of Theorems~\ref{func-spl} and \ref{vol-area-spl} along with a detailed and rigorous discussion of all the continuity issues in those theorems and a detailed proof of Lemma~\ref{eu-nz}, which is the main tool to obtain all these results. In the other paper~\cite{sob2010} the proof uses much simpler topology (without any analogue of Lemma~\ref{eu-nz}), which is still enough to proof Theorem~\ref{gen-spl} for any number of parts $q$.

Of course, in order to apply Theorem~\ref{gen-spl} to measures in $\mathbb R^d$ we have to first suppose that the measures have compact support; for arbitrary probability measures we may use approximation and carefully go to the limit. The partition into convex parts is still possible, but possibly it will not be described as $P(F)$ for some system of linear functions $F$, that is it may not be a generalized Voronoi partition.

Let us give more results about generalized Voronoi partitions.

\begin{thm}
\label{func-spl}
Suppose $C\subset\mathbb R^n$ is a convex body, $\mu$ is an absolutely continuous probability measure on  $C$, $\phi_1, \ldots, \phi_{n-1}$ are functions of convex compact sets continuous with respect to the Hausdorff metric, and $q$ is a prime power. Then $C$ can be partitioned into $q$ convex parts $V_1,\ldots, V_q$ so that 
$$
\mu(V_1) = \dots = \mu(V_q),
$$
and for every $i = 1,\ldots, n-1$
$$
\phi_i(V_1) = \dots = \phi_i(V_q).
$$
\end{thm}

In particular, we can take $\phi_i(K)$ to be $i$-th Steiner measure, i.e. the coefficient at $t^i$ in the polynomial ($B$ is the unit ball here)
$$
P_K(t) = \mu(K+tB).
$$

As a particular case, we obtain the following fact: any compact convex set $C\subset \mathbb R^2$ can be partitioned into $q$ convex parts with equal areas and perimeters, provided $q$ is a prime power. Such results were conjectured and proved in particular cases $q=3, 4$ in~\cite{nara2008,barblsz2010}, the smallest remaining open case in this question is therefore $q=6$.

A similar theorem holds for the standard $n$-dimensional sphere and its convex subsets:

\begin{thm}
\label{func-spl-sph}
Suppose $C\subset S^n$ is a convex body, $\mu$ is an absolutely continuous probability measure on $C$, $\phi_1, \ldots, \phi_{n-1}$ are functions of convex compact subsets of $S^n$ continuous with respect to the Hausdorff metric, and $q>1$ is a prime power. Then $C$ can be partitioned into $q$ convex parts $V_1,\ldots, V_q$ so that 
$$
\mu(V_1) = \dots = \mu(V_q),
$$
and for every $i = 1,\ldots, n-1$
$$
\phi_i(V_1) = \dots = \phi_i(V_q).
$$
\end{thm}

The next theorem does not follow directly from Theorem~\ref{func-spl} because of some discontinuity issues, but is proved in a similar manner. This is a higher-dimensional generalization of the results about perfect partitions in the plane, see~\cite{akknrtu2004}.

\begin{thm}
\label{vol-area-spl} 
Suppose $C\subset\mathbb R^n$ is a convex body, and for some $1\le k \le n$ we have $k$ absolutely continuous probability measures $\mu_1,\ldots, \mu_k$ on $C$, and $n-k$ absolutely continuous probability measures $\sigma_1,\ldots, \sigma_{n-k}$ on $\partial C$. Then for any $q\ge 1$ the body $C$ can be partitioned into $q$ convex parts $V_1,\ldots, V_q$ so that for any $i=1,\ldots, k$
$$
\mu_i(V_1) = \dots = \mu_i(V_q),
$$
and for every $i = 1,\ldots, n-k$
$$
\sigma_i(V_1\cap\partial C) = \dots = \sigma_i(V_q\cap\partial C).
$$
\end{thm}

It would be interesting to generalize the above theorems in the following direction. Let us prescribe positive reals $\alpha_1,\ldots, \alpha_q$ with $\alpha_1+\dots+\alpha_q = 1$, several probability measures $\mu_1, \ldots, \mu_k$ and try to find a convex partition $V_1, \ldots, V_q$ of $\mathbb R^n$ so that for every $1\le i\le k$ and $1\le j\le q$
$$
\mu_i(V_j) = \alpha_j.
$$
In~\cite{aah1998} (and reproved in~\cite{zivvre2001,kar2005}) such a result was established for $k=1$ and any $q$. The partition had the form $P(F)$, where $F$ is a set of linear functions with prescribed and distinct degree $1$ homogeneous parts (and variable free terms). In~\cite{kaka2002} a similar result was established for two measures in $\mathbb R^2$ of special kind, the first being the standard area in a convex body $K$ and the second being the length measure on $\partial K$. It seems that for $d\ge 3$ and $k\ge 2$ a convex partition is not sufficient and it makes sense to consider non-convex partitions. In~\cite{stwo1985,alon1987} it was shown that in the one-dimensional case it is enough to consider partitions into unions of segments with complexity bounded by $q$ and $k$ (approximately by the product of $q$ and $k$), see also Section~\ref{one-dim-sec}. 

Using the same technique we slightly generalize the Borsuk--Ulam type theorem of Gromov and Memarian~\cite[Theorem~3]{mem2009}. The difference is that $q$ is not required to be a power of two, but can be any prime power, and we partition several measures into equal parts at the same time. 

First, we have to define a general notion of a center function.

\begin{defn}
Let $L\subset C(X)$ be a finite-dimensional linear subspace of functions. Suppose that for any subset $F\subset L$ such that all sets $\{V_1,\ldots,V_q\} = P(F)$ have nonempty interiors we can assign \emph{centers} $c(V_1),\ldots, c(V_q)\in X$ to the sets. If this assignment is continuous w.r.t. $F$ and equivariant (with respect to the permutations of functions in $F$ and permutations of points in the sequence $c_1, \ldots, c_q$), we call $c(\cdot)$ a \emph{$q$-admissible center function} for $L$.
\end{defn}

\begin{thm}
\label{gromov-gen}
Suppose $L$ is a measure-separating subspace of $C(X)$ of dimension $n+1$, $\mu_1,\ldots, \mu_{n-k}$ $(n>k)$ are absolutely continuous (with respect to the original measure on $X$) probability measures on $X$, $c(\cdot)$ is a $q$-admissible center function for some prime power $q$, and 
$$
h: X\to \mathbb R^k
$$
is a continuous map. Then there exists a $q$-element subset $F\subset L$ such that for every $i=1,\ldots,n-k$ the partition $P(F)$ partitions the measure $\mu_i$ into $q$ equal parts, and we have 
$$
h(c(V_1)) = h(c(V_2)) =\dots = h(c(V_q))
$$
for $\{V_1, \ldots, V_q\} = P(F)$.
\end{thm}

\begin{rem}
\label{pow2-rem}
In the case $q=2^\ell$ in all the above theorems, when we want a convex partition, the partition may be chosen to be a binary space partition by hyperplanes. In this case instead of using the configuration space $F_q(L)$ in Lemma~\ref{eu-nz} we may follow~\cite{mem2009}: Take the first measure $\mu_1$ and parameterize the binary equipartitions of $\mu_1$ by the product of spheres $Q_q(\mathbb R^n) = (S^{n-1})^{\times q-1}$ taking the normals to the partitioning  hyperplanes. The space $Q_q(\mathbb R^n)$ is equal to $\widetilde M(n, \ell)$ in the notation of~\cite[Definition~1.1]{hung1990}. Then note that partitioning of the remaining measures (or functions) into equal parts is guaranteed by non-vanishing of the Euler class (see the definitions in Section~\ref{eu-nz-sec}) $e(\alpha_q)^{n-1}$ in the cohomology $H_{\Sigma_q^{(2)}}^{(q-1)(n-1)} (Q_q(\mathbb R^n); \mathbb F_2)$ (the group $\Sigma_q^{(2)}$ is the $2$-Sylow subgroup of the permutation group $\Sigma_q$). The latter fact is an analogue of Lemma~\ref{eu-nz}, but its proof is much simpler: It is sufficient to construct a $\Sigma_q^{(2)}$-equivariant map $Q_q(\mathbb R^n)\to \alpha_q^{n-1}$ (the \emph{test map}) with a unique non-degenerate $\Sigma_q^{(2)}$-orbit of zeros. This test map is actually defined in (\ref{proj-av}--\ref{test-map}), if we assume the standard $\Sigma_q^{(2)}$-equivariant inclusion $Q_q(\mathbb R^n)\subset F_q(\mathbb R^n)$ (as in~\cite[Lemma~1.6]{hung1990}).
\end{rem}

{\bf Acknowledgments.}
The author thanks Arseniy~Akopyan, Boris~Aronov, Pavle Blagojevi\'c, Fred~Cohen, Alfredo~Hubard, Gabriel~Nivasch, and Alexey~Volovikov for discussions, useful remarks, and references.

\section{Proof of Theorem~\ref{gen-spl}}

The set of all ordered $q$-tuples $F\subset L$ (collections of $q$ pairwise distinct functions) is the configuration space $F_q(L)$, it has the natural action of the symmetric group $\Sg_q$. Denote $\alpha_q$ the $(q-1)$-dimensional representation of $\Sg_q$, this is the subspace of vectors in $\mathbb R^q$ with zero coordinate sum with the action of $\Sg_q$ by permuting the coordinates.

For every $i=1,\ldots, n$ and $P(F) = \{V_1,\ldots, V_q\}$ the values 
$$
\mu_i(V_1) - \frac{1}{q}, \ldots, \mu_i(V_q) - \frac{1}{q}
$$
define a map $f_i : F_q(L)\to \alpha_q$, this map is $\Sg_q$-equivariant, and from the absolute continuity and the measure separation property we deduce that the map $f_i$ is continuous. 

To prove the theorem we have to show that the direct sum map 
$$
f=f_1\oplus\dots\oplus f_n : F_q(L) \to \alpha_q^n
$$
maps some configuration $F$ to zero. 

The representation $\alpha_q$ defines a vector bundle $\alpha_q\times_{\Sg_q} E\Sg_q\to B\Sg_q$ with the orientation sheaf $\pm\mathbb Z\times_{\Sg_q} E\Sg_q$ (here $\pm \mathbb Z$ is the $\Sg_q$-module with the permutation sign action). So it makes sense to consider its Euler class in the cohomology $H^{q-1}(\Sg_q; \pm\mathbb Z)$, which in turn has a natural image in every $H_{\Sg_q}^{q-1}(X; \pm\mathbb Z)$ for every $\Sg_q$-space $X$. Now we use the following Lemma~\cite[Lemma~6]{kar2009} (see also Section~\ref{eu-nz-sec}):

\begin{lem}
\label{eu-nz}
The image of $e(\alpha_q)^n$ is nonzero in the cohomology
$$
H_{\Sg_q}^{(q-1)n}(F_q(\mathbb R^{n+1}); (\pm\mathbb Z)^{\otimes n}) = H_{\Sg_q}^{(q-1)n}(F_q(L); (\pm\mathbb Z)^{\otimes n}).
$$
\end{lem}

\begin{rem}
In the above lemma we may reduce the cohomology coefficients mod $p$, where $q=p^\ell$.
\end{rem}

The nonzero image of $e(\alpha_q)^n$ is naturally interpreted as the nonzero Euler class of the $\Sg_q$-equivariant vector bundle 
$$
\eta : \alpha_q^n\times F_q(L) \to F_q(L),
$$ 
the map $f$ can be interpreted as a section of $\eta$, so it must have a zero.

\section{Proof of Lemma~\ref{eu-nz}}
\label{eu-nz-sec}

Answering the remarks from Pavle~Blagojevi\'c (private communication) and the unknown referee, we provide a proof of Lemma~\ref{eu-nz}. Put $d=n+1$ in this section.

In fact, most important cases of this lemma were previously known. For $q=p$ (i. e. a prime number) this lemma is valid even in $\mathbb Z_p$-equivariant cohomology (if we embed $\mathbb Z_p\subset \Sg_p$ in the natural way). This is a particular case of~\cite[Lemma~5]{kar2009}, and seems to be known much before, see~\cite[Theorem~3.4, Corollaries~3.5 and 3.6]{ct1991} for example. The case $q=2^\ell$ follows from the direct computations in~\cite{hung1990}, reproduced implicitly in~\cite{mem2009} (see also Remark~\ref{pow2-rem}), in this case it holds in $\Sigma_q^{(2)}$-equivariant cohomology mod $2$.

Lemma~\ref{eu-nz} for $d=2$ was actually proved in~\cite{vass1988}. Below we reproduce the proof extended to all $d\ge 2$. We denote $F_q(\mathbb R^d)/\Sg_q$ by $B_q(\mathbb R^d)$ and denote the natural projection $F_q(\mathbb R^d)\to B_q(\mathbb R^d)$ by $\pi$. While this paper was under review, another proof of this lemma (also following~\cite{vass1988}) appeared in the second version of~\cite{arhu2010}. So the reader may find more details and explanations  in~\cite{arhu2010}.

Let us introduce one important construction: by projecting the configuration of $q$ points $F=\{p_1, \ldots, p_q\}$ in $\mathbb R^d$ onto the coordinate axes $x_2, \ldots, x_d$ we obtain the average for every $j=2,\ldots, q$
\begin{equation}
\label{proj-av}
m_j(F) = \frac{1}{q}(x_j(p_1) + \dots + x_j(p_q))
\end{equation}
and $q$ numbers
\begin{equation}
\label{proj-dev}
x_j(p_1) - m_j(F), \ldots, x_j(p_q) - m_j(F),
\end{equation}
which constitute a $\Sg_q$-equivariant map $h_j : F_q(\mathbb R^d)\to \alpha_q$. Totally these maps constitute a map
\begin{equation}
\label{test-map}
\tilde h : F_q(\mathbb R^d) \to \alpha_q^{d-1}
\end{equation}
with the zero set $\tilde Z$ consisting of configurations $F=\{p_1, \ldots, p_q\}$ such that $x_j(p_i)$ does not depend on $i$ for $j \ge 2$. In other words, the set $\tilde Z$ consists of configurations with all points lying on a single line parallel to the first coordinate axis.

Now let us remind the notion of the Fuks cellular partition of $F_q(\mathbb R^d)$~\cite{fuks1970}. Consider an oriented graded tree $T$ of height $d$ (levels are numbered from $1$ to $d+1$ from leaves to the root) with $q$ leaves labeled by numbers $i=1,\ldots, q$. If a vertex $v$ on level $j+1$ of this tree has children $w_1, \ldots, w_k$ (in this order) on level $j$ than for every pair of children $w_a, w_b$ with $a<b$ and every labels $i_a$ and $i_b$ on a descendant of $w_a$ and a descendant of $w_b$ respectively we impose the inequality $x_j(p_{i_a}) <x_j(p_{i_b})$ on the coordinates of the configuration $F$. Note that these inequalities together guarantee that the points $p_i$ are pairwise distinct and so to any tree $T$ we associate an open cell $Z_T\subset F_q(\mathbb R^d)$. If we remove labels on the bottom level then we obtain an open cell $\pi(Z_T)$ of $B_q(\mathbb R^d)$. It is easy to note that the dimension of $Z_T$ equals the number of vertices in $T$ minus $1$.

The set $\tilde Z$ described above corresponds to the $\Sg_q$-orbit of the Fuks cell $Z$ corresponding to the tree $T_Z$ with only one branching at level $2$ and labels on the bottom level consistent with the left-right direction. In other words, the set $F=\{p_1, \ldots, p_q\}$ with
$$
p_i = (x_i^1, \ldots, x_i^d)
$$ 
is in $Z$ if and only if 
$$
x_1^j = x_2^j = \dots = x_q^j
$$
for every $j=2,\ldots, d$ and 
$$
x_1^1 < x_2^1 < \dots < x_q^1.
$$
Note that the cell $\pi(Z)$ is the unique open cell of minimal dimension $d + q - 1$ of the Fuks partition of $B_q(\mathbb R^d)$.

In order to prove that $e(\alpha_q)^{d-1}$ is nonzero we have to prove that the homology class of $\pi(Z)$ is nontrivial in the compact support homology $H_{d+q-1}^c (B_q(\mathbb R^d); \pm\mathbb Z)$. We must use the compact support homology (homology of the one-point compactification) because the manifolds $F_q(\mathbb R^d)$ and $B_q(\mathbb R^d)$ are open and the Poincar\'e--Lefschetz duality takes cohomology to the compact support homology. We always need twisted coefficients because for even $d$ the manifold $B_q(\mathbb R^d)$ is oriented and $e(\alpha_q)^{d-1}$ is in the cohomology with twisted coefficients, while for odd $d$ the orientation sheaf of $B_q(\mathbb R^d)$ is $\pm \mathbb Z$ and $e(\alpha_q)^{d-1}$ is in the untwisted cohomology. See~\cite{bm1960} for the discussion of the compact support homology and the corresponding Poincar\'e duality with twisted coefficients (in a very general setting compared to what we need here).

We have to check that $\pi(Z)$ is not annihilated by the boundary map. The cells of the Fuks partition of dimension $d+q$ correspond to the trees $T_1, \ldots, T_{q-1}$ such that $T_k$ has a binary branching $v\to w_a, w_b$ on level $3$, then $w_a$ has $k$ children on level $1$ and $w_b$ has $q-k$ children on level $1$. Each tree $T_k$ corresponds to the $\Sg_q$-orbit of the cell $Y_k$ given by the (in)equalities:
$$
x_1^j = x_2^j = \dots = x_q^j
$$
for $j = 3,\ldots, d$,
$$
y_a = x_1^2 = x_2^2 = \dots = x_k^2 < x_{k+1}^2 = \dots = x_q^2 = y_b,
$$
$$
x_1^1 < \dots < x_k^1 \quad\text{and}\quad x_{k+1}^1 < \dots < x_q^1.
$$

Now it remains to calculate the coefficient at $\pi(Z)$ in $\partial \pi(Y_k)$ (with appropriate coefficient twist). In~\cite{vass1988} this coefficient was shown to be $\binom{q}{k}$ up to sign for $d=2$. The calculations in~\cite{vass1988} are actually applicable to the case $d > 2$ because the coordinates $j=3,\ldots, d$ are the same for all points in $Y_k$ and $Z$ and do not affect anything. Since the proof of~\cite[Theorem~2.5.1]{vass1988} is very brief and not very clear we present the calculations below.

For any $\rho\in \Sg_q$ in order to make correct calculation we have to orient $\rho Z$ so that the map $\rho : Z\to \rho Z$ preserves the orientation. Note that this orientation coincides with the orientation given by the form $dx_*^d\wedge \dots \wedge dx_*^2\wedge dx_1^1\wedge \dots \wedge dx_q^1$ ($x_*^j$ denotes the common value of $x_i^j$) if and only if $\rho$ is an even permutation.

Let us orient $Y_k$ by the form $dx_*^d\wedge \dots \wedge dx_*^3\wedge dy_a\wedge dy_b \wedge dx_1^1\wedge \dots \wedge dx_q^1$. The boundary $\partial Y_k$ corresponds to approaching the equality $x_*^2 = y_a = y_b$ from $y_a < y_b$ and it is therefore oriented by $dx_*^d\wedge \dots \wedge dx_*^3\wedge dx_*^2 \wedge dx_1^1\wedge \dots \wedge dx_q^1$. Denoting by $(-1)^\sigma$ the sign of a permutation $\sigma$ we obtain:
$$
\partial Y_k = \sum_{\sigma\in M_{k, q-k}} (-1)^\sigma \sigma Z,
$$
where the subset $M_{k, q-k}\subset \Sg_q$ consists of permutations $\sigma$ such that
$$
\sigma(1) < \dots < \sigma(k)\quad\text{and}\quad \sigma(k+1) < \dots < \sigma(q).
$$
Note that $|M_{k, q-k}|=\binom{q}{k}$. For the homology with twisted coefficients we have to calculate:
\begin{equation}
\label{diff-coeff}
\partial \sum_{\tau\in \Sg_q} (-1)^\tau \tau Y_k = \sum_{\tau\in \Sg_q, \sigma\in M_{k, q-k}}  (-1)^\tau (-1)^\sigma \tau \sigma Z = \binom{q}{k} \sum_{\rho =\tau\sigma\in\Sg_q} (-1)^\rho \rho Z.
\end{equation}
Since $q$ is a power of a prime $p$ it follows that we have the congruence of polynomials in $t$:
$$
(1+t)^q \equiv 1 + t^q \pmod p
$$
and therefore all the binomial coefficients $\binom{q}{k}$ are divisible by $p$. Hence all the coefficients of the boundary operator at $\sum_{\rho\in \Sg_q} (-1)^\rho \rho Z$ are divisible by $p$ and $\pi(Z)$ (which is actually the equivariant cycle $\sum_{\rho\in \Sg_q} (-1)^\rho \rho Z$ in the homology with twisted coefficients) does represent a nonzero homology mod $p$.

\begin{rem}
Note the important thing: In (\ref{diff-coeff}) if we used the untwisted $\mathbb Z$ coefficients and the corresponding cycle $\sum_{\tau\in \Sg_q} \tau Y_k$ without signs then the resulting expression would be different and not divisible by $\binom{q}{k}$.
\end{rem}

\begin{rem}
As is was noted by the referee, it is hard to find a reasonable reference for the Poincar\'e duality in the case of twisted coefficients. In order to overcome this issue we provide the following shortcut:

\begin{itemize}
\item
Put for brevity $M = F_q(\mathbb R^d)$, $\xi : \alpha_q^{d-1}\times M\to M$, and $G=\Sigma_q$. So $\xi$ is a $G$-equivariant vector bundle over the open manifold $M$. The action of $G$ on $M$ changes the orientation (by the sign representation of $G$) if and only if $d$ is odd; it changes the orientation of $\xi$ if and only if $d$ is even.

\item
For a generic $G$-equivariant section $s$ of $\xi$ the zero set $Z_s$ is an embedded submanifold of $M,$ not necessarily compact. From the orientation of $\xi$ (which is the same as the orientation of the normal bundle of $Z_s$) we deduce that the action of $G$ on the orientation of $Z_s$ is the sign representation of $G$.

\item
The standard transversality argument shows that for two generic $G$-equivariant sections $s$ ant $t$ of $\xi$ their zero manifolds $Z_s$ and $Z_t$ are bordant in $[0, 1]\times M$; the bordism is not necessarily compact. This bordism is $G$-equivariant and the orientation is in accordance with the sign action of $G$.

\item
The above facts imply that the class of $Z_s$ (denoted by $[Z_s]$) in the $G$-equivariant compact support homology of $M$ with twisted coefficients $\pm \mathbb Z$ is defined independently of the choice of a generic section $s$.

\item
Now it remains to verify that the homology class $[Z_s]$ does not vanish for a particular section, which was already done (see above) by identifying $[Z_s]$ with the orbit of the lowest-dimensional cell in the Fuks cellular decomposition and providing a direct calculation of the differentials in that decomposition.

\end{itemize}
\end{rem}

\section{Proof of Theorems~\ref{func-spl} and \ref{func-spl-sph}}

The proof follows the proof of Theorem~\ref{gen-spl}, but with the following modifications. Let $L$ be the $(n+1)$-dimensional space of (non-homogeneous) linear functions on $\mathbb R^n$, or the space of homogeneous linear functions on $\mathbb R^{n+1}$ restricted to $S^n$.

For $i=1, \ldots, n-1$, we define the maps 
$$
f_i : F_q(L) \to \alpha_q
$$
as follows. For $F\in F_q(L)$ and $P(F) = \{V_1,\ldots, V_q\}$ put 
$$
m_i(F) = \frac{1}{q} \sum_{j=1}^q \phi_i(V_j(F)),
$$
and
$$
f_i : F \mapsto \big(\phi_i(V_1(F)), \ldots, \phi_i(V_q(F))\big) - \big(m_i(F), \ldots, m_i(F)\big).
$$
Define the map $f_n$ as before
$$
f_n : F\mapsto \left(\mu(V_1) - \frac{1}{q}\mu(C), \ldots, \mu(V_q) - \frac{1}{q}\mu(C)\right).
$$

Note that the maps $f_1,\ldots, f_{n-1}$ are defined only for $F$ such that all the sets $\{V_j(F)\}_{j=1}^q$ (we assume $V_j(F) = V_j(F)\cap C$) are nonempty. Moreover, these maps may be discontinuous. To correct this, consider the closed subset $Z\subseteq F_q(L)$ consisting of configurations $F$ such that $f_n(F)=0$. For $F\in Z$ the sets $V_j(F)$ have equal measures, and therefore they are convex compact sets with nonempty interiors (convex bodies), and they depend continuously (in the Hausdorff metric) on $F$, because their facets depend continuously on $F$. Now assume that the maps $f_1, \ldots, f_{n-1} : Z\to \alpha_q$ are defined according to the above formulas; and extend each map $f_i$ ($1\le i \le n-1$)  separately to a continuous $\Sg_q$-equivariant map $f_i : F_q(L)\to \alpha_q$. This can be done because we extend them from a closed subspace and the target space is the Euclidean space.

Now we can use the Euler class and find a common zero of the maps $f_1, \ldots, f_n$, i.e. the zero of 
$$ 
f_1\oplus\dots\oplus f_n: F_q(L)\to (\alpha_q)^n.
$$ 
The condition $f_n(F)=0$ guarantees that $F\in Z$. That is we are in the range where the maps $f_1,\ldots f_{n-1}$ are defined originally and the result follows.

\section{Proof of Theorem~\ref{vol-area-spl}}

Define the maps ($i=1,\ldots, k$)
$$
f_i : F\mapsto \left(\mu_i(V_1) - \frac{1}{q}\mu_i(C), \ldots, \mu_i(V_q) - \frac{1}{q}\mu_i(C)\right).
$$
they are continuous on the whole $F_q(L)$. Again, let $Z\subset F_q(L)$ consist of configurations $F$ such that $f_i(F)=0$ for all $i=1,\ldots, k$. For $F\in Z$ the sets $V_1(F), \ldots, V_q(F)$ are nonempty and have nonempty interior; of course, we assume $V_j(F) = V_j(F)\cap C$.

Now the maps ($i=k+1, \ldots, n$)
$$
f_i : F\mapsto \left(\sigma_{i-k}(V_1\cap\partial C) - \frac{1}{q}\sigma_{i-k}(\partial C), \ldots, \sigma_{i-k}(V_q\cap\partial C) - \frac{1}{q}\sigma_{i-k}(\partial C)\right)
$$
are defined on $Z$. Note that for $F\in Z$ (and in some neighborhood of $U\supset Z$) any two convex sets $V_j(F), V_l(F)$ are separated by a hyperplane $u_j(x) = u_l(x)$; and since $V_j(F)$ and $V_l(F)$ have nonempty interiors this hyperplane is transversal to $\partial C$. Therefore the sets $V_j(F)\cap\partial C$ depend continuously on $F\in U$ and the rest of the proof for a prime power $q$ is similar to the previous proof.

Thus the case when $q$ is a prime power is done. If $q$ is not a prime power, we may iterate partitions in this theorem.

\section{Proof of Theorem~\ref{gromov-gen}}

Again, the proof follows the proof of Theorem~\ref{gen-spl} with certain modifications.

The first $n-k$ maps $f_i : F_q(L)\to \alpha_q$ are given as before, by the measures $\mu_1,\ldots, \mu_{n-k}$ of the parts $P(F)$. The last $k$ maps
$$
f_{n-k+i} : F_q(L) \to \alpha_q
$$
are given as follows: for $F\in F_q(L)$ and $P(F) = \{V_1(F), \ldots, V_q(F)\}$, consider the coordinate function $x_i$ in the target space of $h$, and put
$$
m_i(F) = \frac{1}{q} \sum_{j=1}^q x_i(h(c(V_j(F)))).
$$
Then define 
$$
f_{n-k+i} : F \mapsto \big(x_i(h(c(V_1(F)))), \ldots, x_i(h(c(V_q(F))))\big) - \big(m_i(F), \ldots, m_i(F)\big).
$$
The maps $f_{n-k+1}, \ldots, f_n$ are defined only for $F$ such that all $V_i(F)$ are nonempty. Since the first $n-k$ conditions 
$$
f_1(F) = \dots = f_{n-k}(F) = 0
$$
define a closed subset $Z\subseteq F_q(L)$ and guarantee that all $V_i(F)$ have nonempty interiors, we can extend the maps $f_{n-k+1}, \ldots, f_n$ from $Z$ continuously and $\Sg_q$-equivariantly to the whole $F_q(L)$, and then apply Lemma~\ref{eu-nz} as above.

\section{Measures on the segment and the complexity of the maximum of several functions}
\label{one-dim-sec}

Recall the ``splitting necklace'' theorem in its continuous version.

\begin{thm}[Noga~Alon~\cite{alon1987}]
\label{alon-spl}
Suppose we are given absolutely continuous probability measures $\mu_1,\ldots, \mu_n$ on a segment $[0,1]$. For an integer $r\ge 2$ put $N=n(r-1)+1$. Then $[0,1]$ can be partitioned into $N$ segments $I_1, \ldots, I_N$, the family $\mathcal F = \{I_i\}_{i=1,\ldots, N}$ can be partitioned into $r$ subfamilies $\mathcal F_1,\ldots, \mathcal F_r$ so that for any $i=1,\ldots, n$ and $j=1,\ldots, r$
$$
\mu_i\left(\bigcup \mathcal F_j\right) = \frac{1}{r}.
$$
\end{thm}

Let us try to reduce Theorem~\ref{alon-spl} to Theorem~\ref{gen-spl}.

Take $L$ to be the set of polynomials of degree $\le n$ on the segment $[0, 1]$. In this case we obtain $q$ polynomials, the sets of the partition $P(F)$ are unions of several segments, and we have to show that the total number of segments does not exceed $n(q-1)+1$. This would follow from the following claim.

\begin{fcon}
\label{pol-sup}
Suppose $f_1,\ldots, f_q$ are polynomials of degree $\le n$, for $x\in\mathbb R$ denote
$$
g(x) = \max\{f_1(x), \ldots, f_q(x)\}.
$$
Then $g(x)$ has $\le n(q-1)$ points of switching between a pair of $f_i$'s.
\end{fcon}

\begin{rem}
The function $g(x)$ is usually called an \emph{upper envelope} of the set of polynomials.
\end{rem}

The case of non-prime-power $r$ in the splitting necklace theorem would follow from this conjecture by iterating the splittings, as in the original proof of Theorem~\ref{alon-spl}.

This conjecture is obviously true as stated for $n=1$ or $q=2$, the latter case gives Theorem~\ref{alon-spl} in case $r=2^k$ by iterating (this is the same as using the ``ham sandwich'' theorem). The case $n=2$ can also be done ``by hand'', ordering the polynomials by the coefficient at $x^2$ and applying induction. But generally Conjecture~\ref{pol-sup} is false. Arseniy~Akopyan has constructed a counterexample for $n=3,\ q\ge 4$ (private communication). An unpublished result of P.~Shor (cited in~\cite{agsh1995}) shows that for $n=4$ the number of ``switch'' points may grow as $\Omega(q\alpha(q))$, where $\alpha(q)$ is the inverse Ackermann function. In~\cite{agsh1995} this problem was studied in a combinatorial setting. The sequence of ``switches'' between $q$ polynomials may be encoded as a word in $q$ letters with some restrictions depending on the degree $n$, such sequences are called \emph{Davenport--Schinzel sequences}. It is known~\cite{agsh1995} that the maximum length of such a word complies with Conjecture~\ref{pol-sup} for $n=1,2$; but it is asymptotically superlinear in $q$ for any fixed $n\ge 3$.

The following fact is known: Theorem~\ref{alon-spl} is tight and the number $n(r-1)+1$ cannot be made less. As a consequence, we obtain the following fact about analytic functions:

\begin{thm}
Suppose $L\subset C^\omega[0, 1]$ is an $(n+1)$-dimensional space of functions, $q$ is a prime power. Then there exist distinct  $f_1,\ldots, f_q\in L$ such that the upper envelope
$$
g(x) = \max\{f_1(x), \ldots, f_q(x)\}
$$
has at least $n(q-1)$ non-analytic points.
\end{thm}

\begin{proof}
If for every subset $\{f_1,\ldots, f_q\}\subset L$ the number of changes of maximum in $g(x)$ from $f_i(x)$ to $f_j(x)$ (they are exactly non-analytic points) is less than $n(q-1)$, we would prove Theorem~\ref{alon-spl} using Theorem~\ref{gen-spl} with $<n(q-1)+1$ segments. But this is impossible.
\end{proof}

\end{document}